\documentclass[11pt]{amsart}
\usepackage[a4paper,
            bindingoffset=0in,
            left=1in,
            right=1in,
            top=1.2in,
            bottom=1.1in,
            footskip=.25in]{geometry}

\usepackage{mathpazo}
\usepackage{euler}

\usepackage{quiver}

\usepackage{setspace}
\usepackage[dvipsnames]{xcolor}

\usepackage{tikz}
\usepackage{tikz-cd}

\usepackage{mathtools}

\usepackage{graphicx}
\usepackage{enumerate} 

\usepackage{latexsym}
\usepackage{amssymb,hyperref}
\hypersetup{colorlinks,linkcolor={MidnightBlue},citecolor={Maroon},urlcolor={PineGreen}}
\usepackage[cp850]{inputenc}
\usepackage{epsfig}
\usepackage{amsthm}
\usepackage{amscd}
\usepackage{amsmath}
\usepackage{amsfonts}
\usepackage{graphics}
\newtheorem{theorem}{Theorem}[section]

\newtheorem{lemma}[theorem]{Lemma}
\newtheorem{corollary}[theorem]{Corollary} 
\newtheorem{prop}[theorem]{Proposition}

\newtheorem*{theorem*}{Theorem} 
\newtheorem*{corollary*}{Corollary}
\theoremstyle{definition}

\newtheorem{remark}[theorem]{Remark}
\newtheorem{definition}[theorem]{Definition}
\newtheorem*{remark*}{Remark}

\newtheorem*{definition*}{Definition}
\newtheorem*{example*}{Example}
\newtheorem{conj}[theorem]{Conjecture}
\newtheorem{construction}[theorem]{Construction}

\newtheoremstyle{named}{}{}{\itshape}{}{\bfseries}{}{.0em}{\thmnote{#3}}
\theoremstyle{named}

\newcommand{\BC}{\mathbb C} 
\newcommand{\BR}{\mathbb R} 
 
 \newcommand{\BZ}{\mathbb Z}
\newcommand{\BF}{\mathbb F} 
 
\newcommand{\BP}{\mathbb P} \newcommand{\BG}{\mathbb G}

\newcommand{\CE}{\mathcal E} \newcommand{\CF}{\mathcal F}
\newcommand{\CG}{\mathcal G} \newcommand{\CH}{\mathcal H}
 
 \newcommand{\CL}{\mathcal L}
 
\newcommand{\CO}{\mathcal O} 
 
\newcommand{\CS}{\mathcal S} 
\newcommand{\CU}{\mathcal U} 
 
\newcommand{\CY}{\mathcal Y} 

\newcommand\smvee{\raise0.9ex\hbox{$\scriptscriptstyle\vee$}}
\newcommand{\Eff}{\overline{\text{Eff}}}
\newcommand{\Nef}{\text{Nef}}

\DeclareMathOperator{\Mor}{Mor}

\DeclareMathOperator{\rank}{rank}

\DeclareMathOperator{\Pic}{Pic}

\DeclareMathOperator{\codim}{codim}

\newcommand{\comment}[1]{}

\title[Genus One Curves on del Pezzo Threefolds]{Moduli Space of Genus One Curves on Quartic and Quintic del Pezzo Threefolds}

\author{Enhao Feng}
\address{Boston College, Chestnut Hill, 02467, MA, USA}
\email{fenge@bc.edu}
\author{Fumiya Okamura}
\address{The Institute of Science and Engineering, Chuo University, 
1-13-27 Kasuga, Bunkyo-ku, Tokyo 112-8551, Japan}
\email{fokamura941@g.chuo-u.ac.jp}

\begin{document}
\setstretch{1.15}
\begin{abstract}
    In this article, we study the space of smooth genus one curves on del Pezzo threefolds of degree $4$ and $5$. We describe the irreducible components of the Kontsevich moduli space generically parametrizing genus one stable maps with irreducible domains and classify the irreducible components of the morphism space from general elliptic curves. Our result verifies Geometric Manin's conjecture for all del Pezzo threefolds of degree $4$ and $5$ over the complex numbers.
\end{abstract}
\maketitle

\section{Introduction}\label{section: intro}

Fano varieties are a captivating class of varieties because of their rich geometry and arithmetic, and one intriguing problem is to study the moduli space of curves on them. Recent advancement of Geometric Manin's Conjecture \cite{LRT25b} provides an insightful approach to the classification of irreducible components of these moduli spaces. While many results have been concerning the space of rational curves (\cite{Thomsen1998, KP01, Harris2004, Castravet2004, Testa2005, Testa2009, Coskun2009, Riedl2019, Lehmann2019, Lehmann2021a, Shimizu2022, Burke2022, Beheshti2023, Okamura2024dP, JovinellyOkamura2024coindex3, Okamura2025AM} etc), in this paper we classify the space of genus one curves on certain del Pezzo threefolds of Picard rank one, i.e. three dimensional smooth Fano varieties with $\Pic X = \BZ H$ and $-K_X = 2H$, extending techniques in \cite{Fen26}. We work over the complex numbers throughout the paper.

Our first result classifies the irreducible components of the morphism space. Let $C$ be a smooth projective curve and $X$ be a smooth projective variety. For each numerical equivalence curve class $\alpha$, we denote by $\Mor(C,X,\alpha)$ the quasi-projective scheme parametrizing morphisms $f:C\to X$ satisfying $f_*C = \alpha$. When $X$ is a del Pezzo threefold with $\Pic X = \BZ H$, we also let $\Mor(C,X,e)$ to denote $\Mor(C,X,\alpha)$ for $\alpha$ satisfying $H\cdot\alpha = e$. We say $f$ is a \textit{free curve }(resp. a \textit{very free curve}) if $f^*T_X$ is generically globally generated and $h^1(C, f^*T_X) = 0$ (resp. $f^*T_X(-p)$ is generically globally generated and $h^1(C, f^*T_X(-p)) = 0$ for a general $p\in C$).

\begin{theorem}
    \label{irredMor} Let $X$ be a del Pezzo threefold with $\Pic(X) = \BZ H$ and $H^{3}\in \{4,5\}$. Let $E$ be an elliptic curve general in moduli. Then for $e\geqslant 6$, the space $\Mor(C, X, e)$ consists of exactly two irreducible components, of which one generically parametrizes free curves, and the other parametrizes covers of lines. 
\end{theorem}

The above theorem is, in fact, a consequence of the following result which classifies the components of the Kontsevich moduli space $\overline{M}_{g,n}(X)$. Let $\widetilde{M}_{g,n}(X,e) \subset \overline{M}_{g,n}(X,e)$ denote the union of components which generically parametrize $n$-pointed genus $g$ stable maps of $H$-degree $e$ with irreducible domains. 

\begin{theorem}\label{irredKontsevich}
    Let $X$ be a del Pezzo threefold with $\Pic(X) = \BZ H$ and $H^{3}\in \{4,5\}$. 
    Then for each $e\geqslant 5$, the space $\widetilde{M}_{1,0}(X,e)$ consists of two irreducible components $R_{e}$ and $N_{e}$, where 
    \begin{itemize}
        \item $R_{e}$ is a $2e$-dimensional family generically parametrizing free birational stable maps with smooth domains; and
        \item $N_{e}$ is a $2e+2$-dimensional family parametrizing multiple covers of lines. 
    \end{itemize}
    When $H^3 = 4$, the space $\widetilde{M}_{1,0}(X,4)$ consists of $R_4$, $N_4$, and an extra component $Q_4$ parametrizing double covers of conics. When $H^3 = 5$, the space $\widetilde{M}_{1,0}(X,4)$ consists of only $N_4$ and $Q_4$. 
\end{theorem}

The analogue of the above theorem in the case of rational curves is known by \cite[Theorem 7.9]{Lehmann2019}, whose proof relies on inductively showing the irreducibility of lower degree components. Our approach follows along the same vein, and one of the key steps is the following Bend-and-Break result in the genus one setting:

\begin{theorem} \label{Intro:MBB} 
    Let $X$ be a del Pezzo threefold with $\Pic(X) = \BZ H$ and $H^{3}\in \{4,5\}$. 
    For $e\geqslant H^{3}+1$, any free component $M\subset\widetilde{M}_{1,0}(X,e)$ contains a stable map $f: Z = Z_1 + Z_2 \to X$ satisfying the following conditions:
    \begin{itemize}
        \item $Z_1$ is a smooth genus one curve and $f|_{Z_1}$ is free.
        \item $Z_2$ is a rational curve and $f|_{Z_2}$ is free.
        \item $f|_{Z_i}$'s are birational onto their images.
\end{itemize}
\end{theorem}

\subsection{Geometric Manin's Conjecture} \label{section: GMC}

Our strategy for the problem has its origin in the classical Manin's conjecture \cite{Franke1989} which concerns the counting of rational points of bounded height on Fano varieties. In the global function field setting of the conjecture, i.e. when $X$ is a Fano variety and $C$ is a smooth projective curve over $K = \BF_q$, the $K(C)$-points of $X$ correspond exactly to the $\BF_q$-points of the morphism space $\Mor(C,X)$, and an influential heuristic due to Batyrev \cite{Bat88} provides an asymptotic formula for the number of $\BF_q$-points on $\Mor(C,X)$. 

However, the heuristic relies on an optimistic, yet often false, assumption that given a fixed numerical class $\alpha$, the space $\Mor(C,X,\alpha)$ is irreducible of the expected dimension and parametrizes a dominant family of curves. For this reason, Lehmann and Tanimoto proposed the Geometric Manin's Conjecture in \cite{Lehmann2019} that distinguishes the irreducible components of $\Mor(C,X)$ into two types, the \textit{Manin components} and the \textit{exceptional components}. Generally speaking, they expect that to obtain the correct asymptotic formula predicted by Manin's conjecture, one should count rational points on the Manin components and ignore the contributions coming from the exceptional components. Moreover, they observed that the heuristic for $\Mor(C,X)$ is not restricted to the finite field setting, but also makes sense when $X$ is defined over $\BC$. The latter is the setting we are interested in, and we explain the expectations for these two types of components.

Firstly, Manin components are predicted to have the expected dimensions and generically parametrize free curves. In particular, the number of Manin components is conjectured to stabilize to the size of the Brauer group:

\begin{conj}[Geometric Manin's Conjecture (IV), \cite{LRT25b} Section 4] \label{GMC4} Let $X$ be a smooth Fano variety over $\BC$ with Brauer group $Br(X)$. Let $C$ be a smooth connected genus $g$ curve. There exists a curve class $\beta \in Nef_1(X)_{\BZ}$ such that for all $\alpha \in \beta + Nef_1(X)_{\BZ}$, there are exactly $|Br(X)|$ many Manin components in $\Mor(C, X,\alpha)$.
\end{conj}

The condition that the curve classes are sufficiently interior to the nef cone in Conjecture \ref{GMC4} is necessary. For example, the space $\widetilde{M}_{1,0}(X,4)$ in Theorem \ref{irredKontsevich} has an extra component parametrizing double covers of conics, and this will contribute at least one additional Manin component in degree $4$. When $X$ is a del Pezzo threefold of Picard rank one, the Brauer group $Br(X)$ is trivial, and Theorem \ref{irredMor} implies the Geometric Manin's Conjecture for quartic and quintic del Pezzo threefolds and general elliptic curves. 

On the other hand, the exceptional components usually exhibit pathological behaviors, such as having higher than the expected dimensions or having universal families whose evaluation maps have disconnected fibres. Their existence is expected to be controlled by a birational invariant of $X$:

\begin{definition}[$a$-invariant]\label{intro dfn: a-inv}
    Let $X$ be a smooth projective variety and $L$ a nef and big divisor on $X$. The $a$-\textit{invariant} of the pair $(X,L)$ is defined as 
    \[
    a(X,L) = \inf\ \{t\in \mathbb{R}\mid K_{X} + tL \in \overline{\mathrm{Eff}}^{1}(X)\}. 
    \]
\end{definition}

Generally speaking, if $M \subset \Mor(C,X)$ is an exceptional component, then $M$ is expected to be induced by an irreducible component $N\subset \Mor(C,Y)$ along some morphism $f:Y \to X$, where $Y$ either has a smaller dimension than $X$, or is equipped with a generically finite dominant non-birational morphism $f: Y \to X$ satisfying $a(Y,-f^*K_X) \geqslant a(X,-K_X)$. When $X$ is a del Pezzo threefold with $\Pic(X) = \BZ H$ and $H^{3}\in \{4,5\}$, we will show in Section \ref{section: exceptional} that every exceptional component parametrizes a family of maps factoring through the universal family of lines $\CU$ on $X$, and each map in the family is a cover of a line.

\subsection{Previous works}
In addition to \cite{Fen26}, which classified the space of elliptic curves on smooth cubic threefolds and served as the starting point of this work, there are several other algebraic varieties on which genus one curves are studied. For Grassmannian, the morphism space is explored in \cite{Bru87} and the Kontsevich space is analyzed in \cite{Qua23}. For certain homogeneous spaces, the connectedness of the Kontsevich space is studied in \cite{KP01}, and \cite{Per12} and \cite{PP13} addressed the irreducibility of the morphism space.

\smallskip
\noindent \textbf{Leitfaden. } 
In Section \ref{section: prelim}, we collect preliminary results on del Pezzo threefolds, vector bundles on elliptic curves, deformation theory of curves, and the $a$-invariant. 
In Section \ref{section: exceptional}, we classify all exceptional components of $\mathrm{Mor}(E,X)$, where the Hilbert schemes of lines play a central role. 
We classify the non-free components in Section \ref{section: nonfree} using Grauert-M\"{u}lich type theorem (cf. \cite[Section 4]{Fen26}). 
In Section \ref{section: lowdeg}, we study low degree curves of genus one on del Pezzo threefolds. 
In particular, we show that for each degree $e\geqslant H^{3}$, there exist elliptic curves of degree $e$ contained in del Pezzo threefolds with $H^{3}\in \{4,5\}$. 
In the last Section $\ref{section: free}$, we use an inductive argument to prove the main theorems of the paper. 

\smallskip
\noindent \textbf{Acknowledgments. } 
The first author would like to thank Brian Lehmann for helpful discussions and, more importantly, for his constant encouragement throughout the years. He would also like to thank Eric Jovinelly for many useful discussions related to the topic. The second author was supported by JSPS KAKENHI Grant Number JP25KJ0322.

\section{Preliminaries}\label{section: prelim}

Let $X$ be a smooth projective variety. 
We denote by $N_{1}(X)$ (resp. $N^{1}(X)$) the set of numerical classes of $1$-cycles (resp. divisors) on $X$ (with $\BR$-coefficients). 
Let $\Eff^{1}(X) \subset N^{1}(X)$ be the closure of the cone generated by effective divisors on $X$. 
Let $\Nef_{1}(X)\subset N_{1}(X)$ be the closed cone consisting of $1$-cycles $\alpha$ such that $D\cdot \alpha\geqslant 0$ for any $D\in \Eff^{1}(X)$. 

For a smooth projective connected curve $C$ and a $1$-cycle class $\alpha \in N_{1}(X)$, let $\mathrm{Mor}(C,X,\alpha)$ be the quasi-projective scheme parametrizing morphisms $f\colon C\rightarrow X$ such that $f_{*}C = \alpha$. 
We set 
\[
\mathrm{Mor}(C,X,\alpha) = \bigsqcup_{\alpha\in N_{1}(X)}\mathrm{Mor}(C,X,\alpha).
\]
For non-negative integers $g$ and $n$, we let $\overline{M}_{g,n}(X,\alpha)$ be the projective coarse moduli scheme parametrizing (isomorphism classes of) $n$-pointed stable maps $(f\colon C\rightarrow X, p_{1},\dots, p_{n})$ of genus $g$ on $X$ such that $f_{*}C = \alpha$. 
We set 
\[
\overline{M}_{g,n}(X) = \bigsqcup_{\alpha\in N_{1}(X)}\overline{M}_{g,n}(X,\alpha),
\]
and we denote by $\widetilde{M}_{g,n}(X, \alpha)\subset \overline{M}_{g,n}(X,\alpha)$ (resp. $\widetilde{M}_{g,n}(X)\subset \overline{M}_{g,n}(X)$) the union of irreducible components which generically parametrize stable maps with irreducible domains. 

\subsection{Del Pezzo Threefold} \label{subsection: dP3fold}
 
A \textit{del Pezzo threefold} of Picard rank one is a smooth Fano threefold $X$ with $\Pic(X) = \BZ H$ and $-K_X = 2H$. Such threefolds are completely classified:

\begin{theorem}[\cite{Iskovskikh1999}] Let $X$ be a del Pezzo threefold with $\Pic(X) = \BZ H$, then $1\leqslant H^3 \leqslant 5$ and $X$ is one of the following:
\begin{enumerate}[(i)]
    \item $H^3 = 1$: $X$ is a degree $6$ hypersurface in the weighted projective space $\BP(3,2,1,1,1)$.
    \item $H^3 = 2$: $X$ is a double cover of $\BP^3$ ramified along a smooth quartic surface.
    \item $H^3 = 3$: $X$ is a smooth cubic threefold.
    \item $H^3 = 4$: $X$ is a complete intersection of $2$ quadrics in $\BP^5$.
    \item $H^3 = 5$: $X$ is a section of the Grassmannian $\BG(1,4)$ by a general linear subspace of codimension $3$.
\end{enumerate}
\end{theorem}

We will refer to the del Pezzo threefold in case (iv) (resp. (v)) as a quartic (resp. quintic) del Pezzo threefold. The following results regarding the explicit geometry of the Fano surface of lines on $X$ will play an important role in Section \ref{section: exceptional} when we analyze the exceptional components.

\begin{theorem}[\cite{Newstead1968}]\label{quartic del pezzo equiv to genus two curve}
    There is a one-to-one correspondence between isomorphism classes of smooth genus two curves and projective equivalence classes of pencils of quadrics in $\BP^5$.
\end{theorem}

\begin{theorem}[\cite{Reid1972} Theorem 4.8]\label{FanosurfacedP4}
    Let $X$ be a quartic del Pezzo threefold. The Fano surface of lines $V$ of $X$ is isomorphic to the Jacobian $J(C)$ of the unique smooth genus $2$ curve $C$ under the correspondence of Theorem \ref{quartic del pezzo equiv to genus two curve}.
\end{theorem}

\begin{theorem}[\cite{FN89} Theorem I, Lemma 2.1] \label{FN89} 
    Let $X$ be a quintic del Pezzo threefold. The Fano surface of lines $V$ of $X$ is isomorphic to $\BP^2$ and the universal family $\CU$ is isomorphic to the $\BP^1$-bundle $\BP_{\BP^2}(\CE_{0})$, where $\CE_{0} \cong \tau_*\CO_{Q}(4F_1)$ is a rank $2$ vector bundle, $\tau: Q = \BP^1\times\BP^1 \to \BP^2$ is a double cover branched along a smooth conic, and $F_1$ is the fibre of the first projection map $\pi_1: Q \to \BP^1$. Moreover, we have $\det(\tau_*\CO_Q(kF_1)) = \CO_{\BP^2}(k-1)$ and $\CO_{\CE_{0}}(1)$ is nef.
\end{theorem}

\begin{lemma} [\cite{FN89} Lemma 2.2, Lemma 2.3]\label{FN90Lemma2.2} 
    Notation as in Theorem \ref{FN89}. There is a morphism $\psi: \CU \to \BP^6$ given by a sublinear series of the line bundle $\CO_{\CE_{0}}(1)\otimes\pi^*\CO_{\BP^2}(1)$ whose image $\psi(\CU)$ is isomorphic to $X$. In particular, the evaluation map factors through $\psi$.
\end{lemma}

The Hilbert scheme of conics on such $X$ is also irreducible:

\begin{theorem}[\cite{KPS18} Proposition 2.3.8]\label{hilb scheme of conics}
    Let $X$ be a del Pezzo threefold with $\Pic(X) = \BZ H$ and $H^3\in\{3,4\}$. Then the Hilbert scheme of conics $\overline{\CH}^{2,0}$ of $X$ is smooth and irreducible. 
\end{theorem}

The components for the moduli space of rational curves on such del Pezzo threefolds has been completely classified, and we will need the following theorem in the proof of our main theorems in Section \ref{section: free}.

\begin{theorem}[\cite{Lehmann2019} Theorem 7.9; \cite{Shimizu2022} Theorem 1,1]\label{thm: LT19 Thm 7.9}
    Let $X$ be a smooth del Pezzo threefold with $\Pic(X) = \BZ H$. 
    When $H^{3} \leqslant 2$, assume $X$ is general in moduli.
    Then for $H$-degree $e\geqslant 2$, $\overline{M}_{0,0}(X,e)$ consists of two irreducible components:
    \[
    \overline{M}_{0,0}(X,e) = R^{rat}_e \cup N^{rat}_e
    \]
    such that $R^{rat}_e$ generically parametrizes maps that are birational onto their images and $N^{rat}_e$ parametrizes degree $e$ covers of lines.
\end{theorem}

\subsection{Vector bundles on elliptic curves}\label{subsec: vbelliptic}

The classification of vector bundles on elliptic curves \cite{Atiyah1957vbelliptic} is particularly useful for explicitly computing the deformation of genus one curves on $X$. We will use results in this section without referencing.

\begin{definition}
    Let $X$ be a smooth projective variety and $\CE$ a vector bundle on $X$. 
    We say that $\CE$ is \textit{globally generated} (resp. \textit{generically globally generated}) if the map $H^{0}(X,\CE)\otimes \CO_{X}\rightarrow \CE$ is surjective (resp. surjective at the generic point of $X$).
\end{definition}

\begin{theorem}[\cite{Atiyah1957vbelliptic}]
    Let $E$ be an elliptic curve and $\CE$ a vector bundle of rank $r$ and degree $d$ on $E$. 
    \begin{itemize}
        \item When $\CE$ is indecomposable, then $\CE$ is globally generated (resp. generically globally generated) if and only if either $\CE$ is the trivial line bundle $\CO_{E}$, or satisfies $d > r$ (resp. $d \geqslant r$). 

        \item When $\CE$ is decomposable, then $\CE$ is globally generated (resp. generically globally generated) if and only if each direct summand is globally generated (resp. generically globally generated). 
    \end{itemize}
    In particular, if $\CE$ is a generically globally generated vector bundle with $h^{1}(E,\CE) = k$, then $\CE$ contains $k$-copies of $\CO_{E}$ as direct summands. 
\end{theorem}

\begin{definition}
    Let $\CE$ be a vector bundle on an elliptic curve $E$. Suppose the Harder-Narasimhan filtration of $\CE$ is given by
    \[
        0 = \CE_0 \subset \CE_1 \subset \dots \subset \CE_k = \CE.
    \]
    We define the \textit{slope panel} $SP(\CE)$ of $\CE$ to be the following:
    \[
        SP(\CE) = (\underbrace{\mu(\CE_1/\CE_0), \dots}_{rk(\CE_1/\CE_0) \ \text{copies}}, \dots, \underbrace{\mu(\CE_k/\CE_{k-1}), \dots}_{rk(\CE_k/\CE_{k-1}) \ \text{copies}})
    \]
\end{definition}

\subsection{Deformation of curves}\label{subsec: deform} 

We briefly record the deformation theory of the morphism space $\Mor(C,X)$ and the Kontsevich space of stable maps $\overline{M}_{g,0}(X)$. Let $C$ be a smooth projective curve and $X$ be a smooth projective variety. 

\begin{definition}
  Let $f:C\to X$ be a non-constant morphism. We define the \textit{normal sheaf} $N_f$ to be the cokernel of the map $T_C \to f^*T_X$.  
\end{definition}

The tangent space of $\Mor(C,X)$ at $[f]$ is given by $H^0(C, f^*T_X)$ and the obstruction space is $H^1(C,f^*T_X)$. On the other hand, if we consider $f$ as an element of the Kontsevich space, then the tangent space of $\overline{M}_{g,0}(X)$ at $[f]$ is given by $H^0(C, N_f)$ and the obstruction space is $H^1(C, N_f)$. 

We will often be looking at family of morphisms whose images go through fixed points in $X$. Specifically, fixing $r$ distinct general points $p_1,\dots,p_r$ on $X$ and let $W \subset \overline{M}_{g,0}(X)$ denote the locus containing deformations of $f$ through the $p_i$'s. Then tangent space of $W$ at $[f]$ is $H^0(C, N_f(-p_1-\dots-p_r))$ and the obstruction space is $H^1(C, N_f(-p_1-\dots-p_r))$. 

\begin{theorem}[\cite{Lehmann2022a} Proposition 2.8] \label{thm: LT22a Prop 2.8}
    Let $W$ be the locus defined above and $[f]$ a general member of $W$. If $W$ parametrizes a dominant family of curves, then $N_f(-p_1-\dots-p_r)$ is generically globally generated. Conversely, if $N_f(-p_1-\dots-p_r)$ is generically globally generated with $h^1(C, N_f(-p_1-\dots-p_r)) = 0$, then $W$ parametrizes a dominant family of curves.
\end{theorem}

\subsection{\textnormal{$a$}-invariant}\label{subsec: ainv}

As we have seen in the introduction, $a$-invariant plays a fundamental role in Geometric Manin's conjecture. In this section, we study surfaces which admit a generically finite morphism to del Pezzo threefold of Picard rank one. 

\begin{definition}\label{dfn: a-inv}
    Let $X$ be a smooth projective variety and $L$ a nef and big divisor on $X$. 
    The \textit{Fujita invariant} (or, $a$-\textit{invariant}) of the pair $(X,L)$ is defined as 
    \[
    a(X,L) = \inf\ \{t\in \mathbb{R}\mid K_{X} + tL \in \overline{\mathrm{Eff}}^{1}(X)\}. 
    \]
    We say that the pair $(X,L)$ is \textit{adjoint rigid} if the Iitaka dimension $\kappa(X, K_{X}+a(X,L)L)$ is equal to zero. 
\end{definition}

When $X$ is a singular projective variety, we define $a(X,L) := a(\widetilde{X}, \phi^{*}L)$, where $\phi\colon \widetilde{X}\rightarrow X$ is a smooth resolution of $X$. This is well defined by \cite[Proposition 2.7]{Hassett2015}. 

\begin{definition}\label{dfn: a-cover}
    Let $X$ be a smooth projective variety and $L$ a nef and big divisor on $X$. 
    We say that $f\colon Y\rightarrow X$ is an $a$-\textit{cover} if $f$ is a dominant generically finite morphism that is not birational such that $a(Y, f^{*}L) = a(X,L)$. 
\end{definition}

The following result classifies projective varieties with large $a$-invariants.

\begin{theorem}
    [\cite{Hor10} 1.3. Proposition; \cite{Fuj89}]\label{a-inv-cor} Let $Z$ be a smooth projective variety of dimension $n$ over an algebraically closed field and let $H$ be a big and nef divisor on $Z$.
    \begin{enumerate}
    \item If $a(Z, H) > n$, then $a(Z, H) = n+1$ and the pair $(Z,H)$ is birationally equivalent to $(\BP^n, L)$ where $L$ is the hyperplane section.
    
    \item If $a(Z, H)  = n$ and $(Z, H)$ is adjoint rigid, then $(Z,H)$ is birationally equivalent to $(Q, L)$ where $Q$ is a quadric hypersurface, possibly singular, and $L$ is the hyperplane section on $Q$.
    
    \item If $a(Z, H) = n$ and $(Z, H)$ is not adjoint rigid, then up to birational equivalence, the canonical map $\pi: Z\to C$ realizes $Z$ as a $\BP^{n-1}$-bundle over a curve $C$ and $L = \CO_\pi(1)$.
    
    \item If $n-1 < a(Z,H) < n$, then $(Z,H)$ is birationally equivalent to $(\BP_{\BP^2}(\CO(2)\oplus\CO^{n-2}), L)$ with $L = \CO_{Z/\BP^2}(1)$. In this case, the pair $(Z,H)$ is adjoint rigid and $a(Z,H) = n - \frac{1}{2}$.
    \end{enumerate}
\end{theorem}

\begin{theorem}[\cite{Beheshti2022} Theorem 4.1; \cite{Burke2022} Theorem 3.10]\label{largea}
    Let $X$ be a del Pezzo threefold with $\Pic(X) = \mathbb{Z}H$. 
    Then 
    \begin{itemize}
        \item there is no subvariety $Y$ with $a(Y, H) > a(X, H) = 2$; 

        \item any $a$-cover $f\colon Y\rightarrow X$ factors rationally through the universal family of $H$-lines on $X$. 
    \end{itemize}
\end{theorem}

We also need to classify surfaces $S\subset X$ such that $(S,H|_{S})$ is adjoint rigid and $a(S,H|_{S})> 1$. 

\begin{prop}[\cite{Fen26} cf. Proposition 3.4]\label{a-inv-prop}
    Let $X$ be a del Pezzo threefold with $\Pic(X) = \BZ H$. Suppose $(S,H|_S)$ is an adjoint rigid surface in $X$ such that $a(S,H|_S)>1$. Then $a(S,H|_{S}) = 3/2$ or $2$, and we have the following: 
    \begin{enumerate}
        \item If $a(S,H|_S) = 3/2$, then $H^{3}\leqslant 2$ and $S$ is the image of $\BP^2$ and we have the following.
            \begin{itemize}
    
                \item If $H^3 = 2$, then $S\in |2H|$.
    
                \item If $H^3 = 1$, then $S\in |4H|$.
            \end{itemize}
    
        \item If $a(S,H|_S) = 2$, then $H^3=1$ and $S\in |2H|$, which is the image of a quadric surface $Q$.
    \end{enumerate}
\end{prop}

\begin{proof} 
Suppose $(S,H|_S)$ is an adjoint rigid surface satisfying $a(S, H|_S) > 1$. Then Theorem \ref{a-inv-cor} implies that $S$ is birational to either $\BP^2$ or a quadric surface $Q$. Let $Y$ denote either $\BP^2$ or $Q$. We have a rational map $f:Y \dashrightarrow S$. 

Let $\nu: S'\to S$ be the normalization and $\psi: Y \dashrightarrow S'$ be a rational map with $f = \nu\circ \psi$. Write $L = \nu^*H|_S$. Denote by $\phi: W \to S'$ a resolution of the indeterminacy of $\psi$. By the proof of \cite[Proposition 3.4]{Fen26}, the first step of running the $(K_W + a(S', L)\phi^*L)$-MMP gives a morphism of smooth surfaces $\phi': W \to W'$ such that $\phi$ factors through $\phi'$. In particular, we may replace $W$ by $W'$ to start with. By repeating this procedure, we obtain a morphism $\widetilde{W} \to Y$, where $\widetilde{W}$ is a minimal model which is a smooth weak del Pezzo surface containing no $(-1)$-curves. We divide the discussion based on the possibility of $Y$:

\begin{enumerate}[(i)]
    \item $Y = \BP^2$ and $a(S,H|_S) = 3/2$: 
    
    Since there is no $(-1)$-curve on $\widetilde{W}$, we have $\widetilde{W} = \BP^2$ and $L = \CO_{\BP^2}(2)$. As the only possible morphism from $\BP^2$ to a normal surface is the identity map, we have $\BP^2 = S'$ and $f$ is a morphism. In particular, $f$ is the normalization map and is thus finite of degree one. Let $n$ be an integer such that $S\in |nH|$. Then we have 
    \[
        4 = L^2 = H|_S^2 = n\cdot H^3.
    \]
    Since both sides of the equation are integers, we must have $H^3\notin \{3,5\}$.

    When $H^{3}=4$, $X$ is a complete intersection of two quadrics in $\mathbb{P}^{5}$ and $S\in |H|$. 
    If $S$ is normal, then $S$ is smooth, hence a del Pezzo surface of degree $4$, which is not isomorphic to $\mathbb{P}^{2}$. 
    If $S$ is non-normal, then it contradicts the fact that any member of $|H|$ has only isolated singularities (cf. \cite[Lemma 3.1]{Lehmann2021a}). 
    Thus $H^{3}\ne 4$. 

    \item $Y = Q$ and $a(S,H|_S) = 2$:
    
    If $Q$ is a smooth quadric surface, we have $\widetilde{W} = Q$. If $Q$ is singular, then $\phi^*L \cdot C = 0$, where $C$ is the $(-2)$-curve obtained from the blowup of the cone point. Hence we can further contract $C$ to obtain a morphism from $\widetilde{W}$ to $Q$. Same degree computation gives
    \[
        2 = L^2 = H|_S^2 = n \cdot H^3.
    \]
    Again, since both sides of the equation are integers, we must have $H^3\notin \{3,4,5\}$. If $H^3 = 2$, then $S\in |H|$, but this contradicts  \cite[Lemma 4.3]{Okamura2025Gt}.
\end{enumerate}
\end{proof}

\begin{theorem}
    \label{Classifya-inv} Let $X$ be a del Pezzo threefold with $\Pic X = \BZ H$ and $H^3\geqslant 3$. If $S$ is a surface contained in $X$, then 
    \begin{enumerate}
        \item either $a(S, H|_S) \leqslant 1$, or
        \item $a(S, H|_S) = 2$, and a resolution $\tilde{S}$ of $S$ factors through the universal family of lines on $X$, i.e. there exists a morphism $g: \tilde{S}\to \CU$, where $\CU$ is the universal family of lines on $X$.
    \end{enumerate}
\end{theorem} 

\begin{proof}
    Suppose $a(S, H|_S) > 1$. Then by Proposition \ref{a-inv-prop}, there is no surface $S \subset X$ such that $(S, H|_S)$ is adjoint rigid with $a(S, H|_S) > 1$, hence we may assume that $(S, H|_S)$ satisfies Theorem \ref{a-inv-cor} (c). Similar argument as Theorem \ref{a-inv-prop} shows that we may replace $S$ by its minimal model and regard it as a $\BP^1$-bundle $\pi: S\to C$. Let $F$ denote the fibre of $\pi$. We have $F \cdot (K_{S} + 2H|_S) = 0$ and $F^2 = 0$. This implies that $F\cdot H|_S = 1$ and the fibres are lines, i.e. $S$ factors through the universal family of lines. 
\end{proof}

\section{Exceptional components}\label{section: exceptional}

Let $X$ be a del Pezzo threefold with $\Pic(X) = \BZ H$, $H^3\in \{4,5\}$. In this section we classify the exceptional components of $\Mor(E,X)$. 

\begin{definition}[cf. \cite{LT26} Definition 4.10]
    Let $X$ be a smooth Fano variety over $\BC$. Let $Y \to X$ be a thin map from a smooth projective variety $Y$, i.e. $f$ is generically finite to the image and there is no rational section of $f$ from $X$. We say $f$ is an \textit{exceptional map} if it satisfies one of the following conditions:
    \begin{itemize}
        \item $f$ is non-dominant and $a(Y, -f^*K_X) \geqslant a(X, -K_X) = 1$;
        \item $f$ is an $a$-cover with $\kappa(K_{Y}-f^{*}K_{X}) > 0$;
        \item $f$ is an $a$-cover with $\kappa(K_{Y}-f^{*}K_{X}) = 0$ that is non-Galois;
        \item $f$ is an $a$-cover with $\kappa(K_{Y}-f^{*}K_{X}) = 0$ that is Galois and face-contracting in the sense of \cite[Definition 4.26]{Lehmann2022};
    \end{itemize}
    We say that an irreducible component $M\subset \mathrm{Mor}(E,X)$ is an \textit{exceptional component} if there is an exceptional map $f\colon Y\rightarrow X$ and an irreducible component $N\subset \mathrm{Mor}(E,Y)$ such that $f$ induces a dominant map $f_{*}\colon N\rightarrow M$. Otherwise, $M$ is a \textit{Manin component}.
\end{definition}

Since there is no $a$-cover with $\kappa(K_{Y}-f^{*}K_{X})=0$ by \cite[Lemma 7.2]{Lehmann2019}, Theorem \ref{largea} implies that any exceptional component on a smooth del Pezzo threefold of Picard rank $1$ factors through the universal family of lines. 

\subsection{The Set-up}
Let $V$ be the Fano surface of lines on $X$ and $\pi\colon \CU\rightarrow V$ be its universal family with the evaluation map $ev:\CU\rightarrow X$. Since $H$ is very ample, we have an embedding of $V$ into the Grassmannian $\mathbb{G}:= \mathbb{G}(1,H^{3}+1)$ of lines on $\mathbb{P}^{H^{3}+1}$. By the universal property, we have the following diagram:
\begin{equation}\label{Diagram: universal family of lines}
    \begin{tikzcd}
    	{\mathcal{U}} & {\mathcal{U}_{\mathbb{G}}} & {\mathbb{P}^{H^{3}+1}} \\
    	V & {\mathbb{G}}
    	\arrow[hook, from=1-1, to=1-2]
    	\arrow["\pi"', from=1-1, to=2-1]
    	\arrow["\lrcorner"{anchor=center, pos=0.125}, draw=none, from=1-1, to=2-2]
    	\arrow["{ev_{\mathbb{G}}}", from=1-2, to=1-3]
    	\arrow["{\pi_{\mathbb{G}}}", from=1-2, to=2-2]
    	\arrow[hook, from=2-1, to=2-2]
    \end{tikzcd}
\end{equation}
Denote by $\mathcal{S}$ the universal rank $2$ subbundle on $\mathbb{G}$ and $\CE \coloneqq \CS^{\vee}|_V$. We have the isomorphisms:
\[
\CU_{\BG} \cong \BP_{\BG}(\CS^{\vee}) \hspace{3mm} \text{and} \hspace{3mm} \CU\cong \BP_V(\CE).
\]
The evaluation map $ev_{\mathbb{G}}$ (resp. $ev$) is defined by a sublinear system of $|\mathcal{O}_{\mathbb{P}(\mathcal{S}^{\vee})}(1)|$ (resp. $|\mathcal{O}_{\mathbb{P}(\mathcal{E})}(1)|$). 
Set $\xi_{\mathbb{G}} := c_{1}(\mathcal{O}_{\mathbb{P}(\mathcal{S}^{\vee})}(1))$ and $\xi := c_{1}(\mathcal{O}_{\mathbb{P}(\mathcal{E})}(1))$. 
Since $\det\mathcal{S}^{\vee} = \mathcal{O}_{\mathbb{G}}(1)$, we see that $\det\mathcal{E} = \det\mathcal{S}^{\vee}|_{V}$ is the very ample line bundle defining the embedding $V\hookrightarrow \mathbb{G}$. 

We consider the locus of morphisms $(f\colon E\rightarrow X)\in \mathrm{Mor}(E,X)$ which factors through the universal family of lines on $X$, i.e. there is a morphism $g\colon E\rightarrow \mathcal{U}$ which fits in the diagram
\[\begin{tikzcd}
	E & {\mathcal{U}} & X \\
	& V
	\arrow["g", from=1-1, to=1-2]
	\arrow["f", curve={height=-18pt}, from=1-1, to=1-3]
	\arrow["h"', from=1-1, to=2-2]
	\arrow["ev", from=1-2, to=1-3]
	\arrow["\pi"', from=1-2, to=2-2]
\end{tikzcd}\]
where $h:= \pi \circ g$. Depending on the image of $h$, we have two cases to consider: 
\begin{itemize}
    \item $h$ contracts $E$ to a point, i.e. $f$ is a multiple cover of a line on $X$, or
    \item $h(E)$ is a curve on $V$. 
\end{itemize} 

\subsection{Cover of lines}\label{subsec: coveroflines}
For the first case, we have the following: 
\begin{theorem}
    \label{coveroflines} Let $e\geqslant 2$. Let $W_E$ denote the locus in $\Mor(E,X)$ that parametrizes degree $e$ covers of lines. Then $W_E$ is an irreducible non-free component of dimension $2e+2$. Moreover, there is a unique irreducible component $N_e \subset \widetilde{M}_{1,0}(X,e)$ of dimension $2e+2$ parametrizing degree $e$ covers of lines.
\end{theorem} 

\begin{proof}
    This follows directly from the proof of \cite[Theorem 4.1]{Fen26} and \cite[Theorem 4.2]{Fen26}.
\end{proof}

\subsection{Non-cover of lines}\label{subsec: noncoveroflines}

We next consider the second case where $h(E)$ is not contracted to a point. 
We will show that such morphisms cannot form an irreducible component of $\mathrm{Mor}(E,X)$. We first prove the following propositions which allow us to bound the expected dimension of $\Mor(E,\CU)$.

\begin{prop} \label{quartic Theta divisor}
    Let $X$ be a quartic del Pezzo threefold. Then $\det \CE$ is algebraically equivalent to $4\Theta$, where $\Theta$ is the Theta divisor on $J(C)$.
\end{prop}

\begin{proof}
    First, we show that $\det \CE$ is represented by a divisor $\pi\circ ev^{-1}(X\cap P)$, where $P\subset \BP^5$ is a general $3$-plane. Indeed, since $\det \CS^\vee = \CO_{\BG}(1)$, diagram \ref{Diagram: universal family of lines} shows that $\det \CE = \CO_{\BG}(1)|_V$. On the other side, $|\CO_{\BG}(1)|$ is the linear series of Schubert cycles $\sigma_P \coloneqq \{l\in \BG\ |\ l\cap P\neq \emptyset\}$. Hence by restriction, $\det \CE$ is represented by the divisor $\pi\circ ev^{-1}(X\cap P)$.

    Next, we show that $\pi\circ ev^{-1}(X\cap P)$ is algebraically equivalent to $4\Theta$. Since $X\cap P$ is a quartic curve, it suffices to show that $\pi\circ ev^{-1}(s)$ is algebraically equivalent to $\Theta$ for a line $s\subset X$. Fix such a line $s$ and denote by $C_s\subset V$ the closure of the set
    \[
    C_s' = \{t\in V\ |\ \dim s\cap t = 0\}.
    \]
    We have $\pi\circ s^{-1}(s) = C_s$. By \cite[Section 4]{Reid1972}, there is an isomorphism $\alpha: V\to J(C)$ factoring through a birational map $\beta: V \dashrightarrow C_s^{(2)}$, where $C_s^{(2)}$ is the symmetric product. The map $\beta$ sends a line $t$ not intersecting $s$ to the point $(s_1,s_2)\in C_s^{(2)}$, where $s_1$ and $s_2$ are residual lines of $X\cap \text{Span}\langle s,t\rangle$. In particular, the locus of indeterminacy of $\beta$ is a finite set of points contained in $C_s$. Denote by $C_s^\circ\subset C_s$ the locus where $\beta$ is defined. Let 
    \[
    u_d: C_s^{(d)} \to J(C)
    \]
    be defined by $u_d(D) = D - ds$. Since the isomorphism $V \cong J(C)$ factors as 
    \[
    \alpha: V\overset{\beta}{\dashrightarrow} C_s^{(2)} \overset{u_2}{\longrightarrow} J(C),
    \]
    we have $\overline{\beta(C_s^\circ)} = \{s\}\times C_s$. On the other hand, the Abel-Jacobi map 
    \[
    u_1: C_s \to J(C)
    \]
    sending $D$ to $D-s$ factors as 
    \[
    C_s \to C_s^{(2)} \overset{u_2}{\longrightarrow} J(C),
    \]
    where the map is given by
    \[
    D \mapsto (D,s) \mapsto D+s-2s = D-s.
    \]
    Since the image $u_1(C_s)$ is a translate of the Theta divisor, and translation preserves algebraic equivalence, the image $\overline{\alpha(C_s^\circ)} \subset J(C)$ is numerically equivalent to $\Theta$.
\end{proof}

When $X$ is a quintic del Pezzo threefold, Theorem \ref{FN89} and Lemma \ref{FN90Lemma2.2} imply that $\mathcal{E} \cong \mathcal{E}_{0}(1) \cong \tau_{*}\mathcal{O}_{Q}(4F_{1})\otimes \mathcal{O}_{\mathbb{P}^{2}}(1)$ and $\det\mathcal{E} \cong \mathcal{O}_{\mathbb{P}^{2}}(5)$. 

\begin{prop} \label{expected dimension on U}
    Let $X$ be a del Pezzo threefold with $\Pic(X) = \BZ H$ and $H^{3}\in \{4,5\}$. Let $g: E\to \CU$ be a morphism such that $h:E\to V$ is not constant. Then
    \[
    K_{\mathcal{U}}\cdot g_{*}E > -2ev^*H\cdot g_*E + 3.
    \]
\end{prop}

\begin{proof}
    Since $\xi = ev^*H$, we have
    \[
     K_{\mathcal{U}}\cdot g_{*}E = -2ev^*H\cdot g_*E + \pi^{*}(K_{V} + \det\mathcal{E})\cdot g_*E
    \]
    by the canonical bundle formula. By Theorem \ref{FN89}, Lemma \ref{FN90Lemma2.2}, and Proposition \ref{quartic Theta divisor}, the second summand is equal to
    \[
    (K_{V} + \det\mathcal{E})\cdot h_{*}E =
    \begin{cases*}
        4\Theta\cdot h_{*}E, & if $H^{3} = 4$, \\
        \mathcal{O}_{\mathbb{P}^{2}}(2)\cdot h_{*}E,  & if $H^{3}=5$. 
    \end{cases*} 
    \]
    Since $h\colon E\rightarrow V$ is non-constant, we have $(K_{V} + \det\mathcal{E})\cdot h_{*}E > 3$, and the result follows.
\end{proof}

\subsubsection{Dominant family of curves on $\CU$}
Firstly, suppose that the deformations of $g\colon E\rightarrow \mathcal{U}$ sweep out $\mathcal{U}$. 
Then we have the inequality
\[
\dim_{[g]}\Mor(E, \CU) \leqslant -K_{\CU} \cdot g_*E + h^1(E, g^*T_{\CU}).
\]
Since $g^{*}T_{\mathcal{U}}$ is generically globally generated, we have $h^1(E, g^*T_{\CU}) \leqslant 3$. 
Hence by Proposition \ref{expected dimension on U}, we have 
\begin{align*}
\dim_{[g]}\Mor(E, \CU) &\leqslant -K_{\CU} \cdot g_*E + 3 \\
 &< 2ev^{*}H\cdot g_{*}E \\
 &\leqslant \dim_{[f]}\mathrm{Mor}(E, X). 
\end{align*}
Therefore, such a dominant family cannot form an irreducible component of $\mathrm{Mor}(E,X)$. 

\subsubsection{Non-dominant family of curves on $\CU$}

Suppose the family sweeps out a surface $\CY\subset \mathcal{U}$. 
There are two cases to consider:

\begin{enumerate}
    \item $\pi(\CY) = h(E)$ is a curve; and  

    \item $\pi|_{\CY}: \CY \to V$ is dominant.
\end{enumerate}

\noindent \textbf{Case (a).} Suppose that $\pi(\CY)$ is one-dimensional. We consider the graph diagram of $g\colon E\rightarrow \mathcal{U}$ over $V$: 
\[\begin{tikzcd}
	E \\
	& {\mathcal{U}_{h}} & {\mathcal{U}} \\
	& E & V
	\arrow["\sigma"{pos=0.6}, from=1-1, to=2-2]
	\arrow["g", curve={height=-6pt}, from=1-1, to=2-3]
	\arrow["{\mathrm{id}}"', curve={height=6pt}, from=1-1, to=3-2]
	\arrow["{h_{\mathcal{U}}}"{pos=0.4}, from=2-2, to=2-3]
	\arrow["{\pi_{h}}"'{pos=0.4}, from=2-2, to=3-2]
	\arrow["\lrcorner"{anchor=center, pos=0.125}, draw=none, from=2-2, to=3-3]
	\arrow["\pi", from=2-3, to=3-3]
	\arrow["h"', from=3-2, to=3-3]
\end{tikzcd}\]
Then $\mathcal{U}_{h}$ is the projective bundle $\mathbb{P}_{E}(\mathcal{F})$ over $E$, where $\mathcal{F} := h^{*}\mathcal{E}$. Let $C := \sigma(E)$ be the section of $\pi_{h}$. The N\'{e}ron-Severi group is given by 
\[
    NS(\mathcal{U}_{h}) \cong \mathbb{Z}\zeta\oplus \mathbb{Z}F,
\]
where $\zeta := c_{1}(\mathcal{O}_{\mathbb{P}(\mathcal{F})}(1)) = h_{\mathcal{U}}^{*}\xi$ and $F$ is the class of fibres of $\pi_{h}$. 
By the projection formula, we have 
\[
    \zeta\cdot C = \xi\cdot g_{*}E\ \text{  and  }\ \zeta^{2} = \deg_{E} \mathcal{F} = c_{1}(\mathcal{E})\cdot h_{*}E.
\]
Since we also have $\zeta\cdot F = C\cdot F = 1$, we see that $C$ and $\zeta$ are numerically equivalent. Thus $C^{2} = \zeta\cdot C = \xi\cdot g_{*}E>0$. 

Since $\pi(\CY)$ is one-dimensional, the space of deformations of $g \in \mathrm{Mor}(E,\mathcal{U})$ corresponds to the space of deformations of section $C \in \mathrm{Sec}(\mathcal{U}_{h}/E)$. 
Hence we have 
\begin{align*}
    \dim_{[g]}\mathrm{Mor}(E, \mathcal{U}) = \dim_{[C]}\mathrm{Sec}(\mathcal{U}_{h}/E)
     \leqslant h^{0}(N_{C/\mathcal{U}_{h}})
     =C^{2}
     &< 2\xi\cdot g_{*}E
     \leqslant \dim_{[f]}\mathrm{Mor}(E,X). 
\end{align*}

\noindent \textbf{Case (b).} Since $\pi|_{\CY}: \CY \to V$ is dominant and generically finite, it induces a generically finite morphism $\mathrm{Mor}(E,\mathcal{U})\rightarrow \mathrm{Mor}(E,V)$ onto its image.
Hence we have $\dim_{[g]}\mathrm{Mor}(E,\mathcal{U}) \leqslant \dim_{[h]}\mathrm{Mor}(E,V)$.
Now $\dim_{[h]}\mathrm{Mor}(E,V) \leqslant -K_{V}\cdot h_{*}E + h^{1}(g^{*}T_{V})\leqslant -K_{V}\cdot h_{*}E + 2$. Then the projection formula and calculation as in Case (a) give 
\begin{align*}
    2H\cdot f_{*}E - (-K_{V}\cdot h_{*}E + 2) &= 2\xi\cdot g_*E + K_{V}\cdot h_{*}E - 2\\
    &= 2c_{1}(\mathcal{E})\cdot h_{*}E + K_{V}\cdot h_{*}E - 2 > 0. 
\end{align*}
Hence $\dim_{[h]}\mathrm{Mor}(E,V) < \dim_{[f]}\mathrm{Mor}(E,X)$.\\ 

Altogether, we have proven:

\begin{theorem} \label{facunifamlinecor}
    Let $X$ be a del Pezzo threefold with $\Pic(X) = \BZ H$ and $H^{3}\in \{4,5\}$. Let $M\subset \Mor(E,X)$ be an irreducible component parametrizing maps factoring through the universal family of lines. Then $M$ parametrizes covers of lines. \qed
\end{theorem}

\section{Non-free components}\label{section: nonfree}

In the previous section, we have shown that the only exceptional components are the ones parametrizing multiple covers of lines. In this section, our aim is to provide a nice geometric characterization of Manin components. In particular, we show that if $M$ is a Manin component parametrizing curves of sufficiently large degree, then a general member of $M$ is a free curve. Let us first recall the definition of a free curve:

\begin{definition}
    Let $C$ be a smooth projective curve and $f: C\to X$ be a morphism. We say that $f$ is a \textit{free curve} (resp. a \textit{very free curve}) if $f^*T_X$ is generically globally generated and $h^1(C, f^*T_X) = 0$ (resp. $f^*T_X(-p)$ is generically globally generated and $h^1(C, f^*T_X(-p)) = 0$ for a general $p\in C$).
\end{definition}

Geometric Manin's Conjecture predicts that a Manin component should parametrize a dominant family of curves. When $C = \BP^1$, this is the same as saying a general member $f:C\to X$ of a Manin component has globally generated tangent bundle, and hence is a free curve. This prompts the following definition:

\begin{definition}
    We say an irreducible component $M \subset \Mor(E,X)$ is a \textit{free component} if $M$ generically parametrizes free curves. Otherwise, we say $M$ is a \textit{non-free component}.
\end{definition}

However, when $g(C) > 0$, a component parametrizing a dominant family of curves may fail to be free, as a generically globally generated tangent bundle can have non-vanishing $h^1$. In such scenario, it can have higher than the expected dimension or can be generically non-reduced. Thus, given a non-free component $M$, it could either parametrize a dominant or a non-dominant family of curves. The following theorem rules out the latter case:

\begin{theorem}[cf. \cite{Fen26} Theorem 4.9]
    Let $X$ be a del Pezzo threefold with $\Pic(X) = \BZ H$ and $H^{3}\in \{4,5\}$. Then there is no non-free component parametrizing a non-dominant family of curves on $X$.
\end{theorem} 

\begin{proof}
    Suppose the contrary that there is an irreducible component $M\subset \Mor(E,X)$ parametrizing a non-dominant family of curves of $H$-degree $d$. By Theorem \ref{coveroflines}, $M$ does not parametrize covers of lines, so we may assume $d \geqslant 3$. Let $Y$ denote a resolution of the closure of the subvariety swept out by curves parametrized by $M$ with a morphism $f: Y \to X$. It is enough to consider the case when $Y$ is a surface.

    Let $N\subset \Mor(E,Y)$ denote the irreducible component induced by $M$, i.e. $f_*: N \to M$ is dominant. Let $s: E\to Y$ be a general member of $N$ with $s_*E = C$. Then we have the dimension bound
    \[
    -K_X\cdot f_*C \leqslant \dim M \leqslant \dim(N) \leqslant -K_Y\cdot C  + 1,
    \]
    where the rightmost inequality follows from \cite[Proposition 4.8]{Fen26}. Let $\epsilon > 0$ be some constant. We can further rewrite the above inequality as 
    \[
    1-(\frac{1}{2} - \epsilon)f^*H\cdot C \geqslant (K_Y + (\frac{3}{2} + \epsilon) f^*H)\cdot C.
    \]
    Since $f^*H\cdot C = d \geqslant 3$,  we may choose $\epsilon$ small enough such that the left hand side of the inequality is less than $0$. By \cite[Theorem 0.2]{Boucksom2013}, the movable cone of curves is dual to the pseudo-effective cone of divisors, hence $K_Y + (\frac{3}{2} + \epsilon)f^*H$ is not big, i.e. $a(Y, f^*H) \geqslant \frac{3}{2} + \epsilon$. Theorem \ref{Classifya-inv} then implies that $a(Y, f^*H) \geqslant 2$, i.e. the map $f:Y\to X$ factors through the universal family of lines on $X$, a contradiction to Theorem \ref{facunifamlinecor}.
\end{proof}

Next, we treat the case when $M$ parametrizes a dominant family of curves. We will focus on the case when the curves parametrized by $M$ have degree at least $5$. The following proposition shows that $M$ generically parametrizes morphisms that are birational onto their images. 

\begin{prop}
    \label{prop: Fen26 Prop 4.6} Let $M$ be an irreducible component of $\Mor(E,X)$ parametrizing a dominant family of curves with degree $e > 4$. Then $M$ either generically parametrizes covers of lines or maps that are birational onto their images.
\end{prop}

Let $s:E\to X$ be a general member of $M$. Then $s^*T_X$ is generically globally generated. We divide our discussion based on the possible decompositions of $s^*T_X$ in the non-free dominant situation:

\begin{enumerate}[\hspace{2mm} (1)]
    \item $s^*T_X = \CO_E \oplus \CF$ and $0 \leqslant \mu^{\min}((N_s)_{tf}) \leqslant 1$. In this case we have $h^1(E, s^*T_X) = 1$.

    \item $s^*T_X = \CO_E \oplus \CF$ and $\mu^{\min}((N_s)_{tf}) \geqslant 2$. In this case we have $h^1(E, s^*T_X) = 1$.

    \item $s^*T_X = \CO_E \oplus\CO_E\oplus\CL$. In this case we have $h^1(E, s^*T_X) = 2$.
\end{enumerate}

\begin{definition}
    Let $X$ be a smooth projective threefold. We say a non-free component $M \subset \Mor(E, X)$ is of \textit{type i} if a general map of $M$ is birational onto its image and satisfies condition $(i)$ above.
\end{definition}

The same arguments as in \cite[Theorem 4.16]{Fen26} and \cite[Theorem 4.22]{Fen26} shows that there are no type 1 components parametrizing maps of degree at least $5$ and no type 3 components parametrizing maps of degree at least $4$ for a del Pezzo threefold $X$ with $\Pic(X) = \BZ H$ and $H^{3}\in \{4,5\}$. In the rest of this section, we focus on type 2 components. In fact, we may further distinguish type 2 components in terms of their non-reduced structures. Let $M$ be a type 2 component and $[s]$ a general member. Consider the short exact sequence:
\[
0 \to T_E \to s^*T_X \to N_s \to 0.
\]
Since $\mu^{\min}((N_s)_{tf}) \geqslant 2$, we have $H^1(E, N_s) = 0$. Hence the map on cohomology $H^1(E, T_E) \to H^1(E, s^*T_X)$ is a bijection. This induces a surjection
\[
H^0(E, s^*T_X) \to H^0(E, N_s).
\]
Since $h^1(E, s^*T_X) = 1$, either $M$ is generically non-reduced and has the expected dimension, or $M$ is generically smooth and has higher than the expected dimension. Let $\overline{M} \subset \overline{M}_{1,0}(X)$ be the irreducible component containing the image of $M$ under the map $\Mor(E,X) \to \overline{M}_{1,0}(X)$. We have the following:

\begin{prop} \label{prop: different type 2 comps}
    Let $M$ be a type 2 component.
    \begin{itemize}
        \item If $M$ is generically non-reduced, then the image of $M$ in $\overline{M}$ has codimension one.
        
        \item If $M$ is generically smooth, then the map $M \to \overline{M}$ is dominant.
    \end{itemize}
\end{prop}

\begin{proof}
    Since the map $H^0(E, s^*T_X) \to H^0(E, N_s)$ is surjective by the paragraph above, the result follows from the isomorphisms $T_{\overline{M}, [s]} \cong H^0(E, N_s)$ and $T_{M, [s]} \cong H^0(E, s^*T_X)$.
\end{proof}

The main tool to show the non-existence of the type 2 components is the Grauert-M\"ulich type theorem (\cite[Theorem 6.5]{LRT25b}). We collect a few numerical results needed to apply such theorem. We omit the proof as it is exactly the same as that of the cited results.

\begin{lemma}[\cite{Fen26} c.f. Proposition 4.13]
    \label{N_s-tor} Suppose that $M$ is a type 2 component. Let $s$ be a general member of $M$. Then the length of the torsion of $N_s$ is at most $1$. 
\end{lemma} 


\begin{lemma}[\cite{Fen26} c.f. Proposition 4.20]
    \label{mu_min at least 4} Let $M$ be a type 2 component parametrizing morphisms of $H$-degree at least $6$. Then we have $\mu^{\min}((N_f)_{tf}) > 3$.
\end{lemma}

\begin{theorem} \label{no type 2 comp}
    There is no generically non-reduced (resp. smooth) type 2 component that parametrizes morphisms of $H$-degree at least $6$ (resp. $5$).
\end{theorem}

\begin{proof}
    Suppose to the contrary that there is such a type 2 component $M$. Let $W$ be a variety equipped with a generically finite morphism onto the locus induced by $M$ in $\overline{M}_{1,0}(X)$. By \cite[Theorem 4.11]{Fen26}, a general fibre of the universal family map $U_W \to W$ is connected, and by \cite[Proposition 4.17]{Fen26}, it lies in the flat locus of the evaluation map $U_W \to X$. Hence we are in the situation of Grauert-M\"ulich \cite[Theorem 6.5]{LRT25b}. 

    Let $s:E\to X$ be a general member of $M$. Write the Harder Narasimhan filtration of $s^*T_X$ as 
    \[
    0 = \CF_0 \subset \CF_1 \subset \dots \subset \CF_k = s^*T_X.
    \]
    Denote by $t$ the length of the torsion part of $N_s$, by $\CG$ the subsheaf of $(N_s)_{tf}$ generated by global sections, by $V$ the tangent space to $W$ at $s$, and by $q$ the dimension of the cokernel of the composition
    \[
    V \to T_{\overline{M}_{1,0}(X),[s]} = H^0(E,N_s) \to H^0(E, (N_s)_{tf}).
    \]
    \cite[Theorem 6.5]{LRT25b} shows that for every index $1\leqslant i \leqslant k-1$, we have
    \[
    \mu(\CF_i/\CF_{i-1}) - \mu(\CF_{i+1}/\CF_i) \leqslant (q+1)\mu^{max}({M_\CG\smvee}) + t, 
    \]
    where $M_\CG$ is the syzygy bundle of $\CG$, i.e. the kernel of the surjective map $H^{0}(X,\CG)\otimes \CO_{X}\rightarrow \CG$. We now divide the discussion into whether $M$ is generically non-reduced or not.

    \begin{itemize}
        \item Suppose $M$ is generically non-reduced and parametrizes maps of degree at least $6$. Then Lemma \ref{mu_min at least 4} shows that $\mu^{\min}(\CG) > 3$, so by \cite[1.3 Corollary]{Butler94}, we have 
        \[
        \mu^{max}({M_\CG\smvee}) \leqslant \frac{\mu^{\min}(\CG)}{\mu^{\min}(\CG) - 1} = 1 + \frac{1}{\mu^{\min}(\CG) - 1} < \frac{3}{2}.
        \]
        Since $t \leqslant 1$ by Lemma \ref{N_s-tor} and $q \leqslant 1$ by the proof of \cite[Lemma 3.7]{LRT25a}, we have
        \[
        \mu(\CF_i/\CF_{i-1}) - \mu(\CF_{i+1}/\CF_i) < (1+1)\cdot\frac{3}{2} + 1 = 4.
        \]
        By assumption, $s^*T_X$ contains $\CO_E$ as a factor. The slope panel $SP(s^*T_X) = (s_1, s_2, s_3) = (s_1, s_2, 0)$. Since $s$ has degree at least $6$, $s_1+s_2 \geqslant 12$. This contradicts the bound on the difference between the slope of successive quotients in the filtration. 

        \item Suppose $M$ is generically smooth and parametrizes maps of degree at least $5$. Then $V\to T_{\overline{M}_{1,0}(X),[s]}$ is surjective and hence $q = 0$. We have \[
        \mu(\CF_i/\CF_{i-1}) - \mu(\CF_{i+1}/\CF_i) \leqslant (0+1)\cdot 2 + 1 = 3.
        \]
        The same argument as above gives a contradiction.
    \end{itemize}
\end{proof}

Altogether we have the following:

\begin{theorem} \label{Thm:no non-free component}
    There is no non-free component generically parametrizing a dominant family of maps of $H$-degree $e\geqslant 6$ that are birational onto their images. In particular, a Manin component parametrizing morphisms of $H$-degree $e\geqslant 6$ is a free component.
\end{theorem}

\section{Low degree curves}\label{section: lowdeg}

In this section, we classify components of the Kontsevich moduli space parametrizing low degree stable maps. This will form the base cases for the induction arguments in Section \ref{section: free}. 

\begin{prop}\label{nodeg3dP4}
    Let $X$ be a quartic del Pezzo threefold. For degree $e\leqslant 3$, there do not exist irreducible components of $\widetilde{M}^{\mathrm{bir}}_{1,0}(X,e)$ generically parametrizing morphisms that are birational onto their images. 
\end{prop}
\begin{proof}
    Since the fundamental line bundle $H$ is very ample, any curve on $X$ of degree at most $2$ has genus $0$. Suppose that $e = 3$.
    If a curve of degree $3$ spans a $\mathbb{P}^{3}$, then it must be rational.
    Hence any genus $1$ curve $C$ of degree $3$ is contained in a unique plane $P$. 
    Since $X$ does not contain any plane, we see that $P\cap X$ is one dimensional.
    Note that $X$ is a complete intersection of two quadric hypersurfaces in $\mathbb{P}^{5}$.
    We take a quadric $Q\subset \mathbb{P}^{5}$ defining $X$ which does not contain $P$. 
    Then we see that $C\subset P\cap X\subset P\cap Q$, a contradiction.
\end{proof}

\begin{prop}\label{deg4dP4}
    Let $X$ be a quartic del Pezzo threefold.
    There is a unique irreducible component in $\widetilde{M}^{\mathrm{bir}}_{1,0}(X,4)$ generically parametrizing morphisms that are birational onto their images.
\end{prop}
\begin{proof}
    It suffices to show that the locus $N$ parametrizing birational stable maps $f\colon E\rightarrow X$ with smooth domains is irreducible.
    As discussed in the proof of Proposition \ref{nodeg3dP4}, the image $C:= f(E)$ is contained in a unique $\mathbb{P}^{3}$. 
    Since the intersection $D:= \mathbb{P}^{3}\cap X$ is a connected curve of arithmetic genus $1$, we must have $C = D$. 
    On the other hand, a general codimension $2$ linear section of $X$ is a smooth curve of degree $4$ and genus $1$. 
    Hence it induces a stable map of $\widetilde{M}^{\mathrm{bir}}_{1,0}(X,4)$ with smooth domain. 
    Thus, we conclude that locus $N$ is birational to the Grassmannian $\mathrm{Gr}(4,6)$.
\end{proof}

\begin{prop}
    \label{nodeg3or4dP5} Let $X$ be a quintic del Pezzo threefold. For degree $e\leqslant 4$, there do not exist irreducible components of $\widetilde{M}^{\mathrm{bir}}_{1,0}(X,e)$ generically parametrizing morphisms that are birational onto their images.
\end{prop} 

\begin{proof}
    Since the fundamental line bundle $H$ is very ample, any curve on $X$ of degree at most $2$ has genus $0$. 
    
    Suppose that $e = 3$.
    Any genus $1$ curve $C$ of degree $3$ is contained in a unique plane $P$. 
    Since $X$ does not contain any plane, we see that $P\cap X$ is one dimensional. Recall that $X$ is defined by quadric hypersurfaces in $\mathbb{P}^{6}$.
    Let $Q\subset \mathbb{P}^{6}$ be a quadric defining $X$ which does not contain $P$. Then we see that $C\subset P\cap X\subset P\cap Q$, a contradiction.
    
    Next, suppose that there exists a birational map $f\colon E\rightarrow X$ from an elliptic curve $E$ to a degree $4$ curve $C$ on $X$. 
    By a similar argument, we see that the curve $C$ is contained in a unique $P\cong\mathbb{P}^{3}$. 
    We take two quadrics $Q_{1}, Q_{2}\subset \mathbb{P}^{6}$ such that $X\subset Q_{1}\cap Q_{2}$ and $P\not\subset Q_{1}\cap Q_{2}$. 
    Since $X$ does not contain any plane and quadric surface by Proposition \ref{a-inv-prop}, $D := P\cap Q_{1}\cap Q_{2}$ is a curve. 
    However, both $C$ and $D$ have degree $4$, hence we must have $C=D$. 
    Since $D$ is a complete intersection curve of arithmetic genus $1$, $D$ is not contained in $X$ by \cite[Proposition 4.4]{CKK25}.
\end{proof}

\begin{prop}
    \label{deg5dP5} Let $X$ be a quintic del Pezzo threefold. There is a unique irreducible component in $\widetilde{M}^{\mathrm{bir}}_{1,0}(X,5)$ generically parametrizing morphisms that are birational onto their images.
\end{prop} 

\begin{proof}
    It suffices to show that the locus $N$ parametrizing birational stable maps $f\colon E\rightarrow X$ with smooth domains is irreducible.
    As discussed in the proof of Proposition \ref{nodeg3or4dP5}, the image $C:= f(E)$ is contained in a unique $\mathbb{P}^{4}$. 
    Since the intersection $D:= \mathbb{P}^{4}\cap X$ is a connected curve of arithmetic genus $1$, we must have $C = D$. 
    On the other hand, a general codimension $2$ linear section of $X$ is a smooth curve of degree $5$ and genus $1$. 
    Hence it induces a stable map of $\widetilde{M}^{\mathrm{bir}}_{1,0}(X,5)$ with smooth domain. 
    Thus, we conclude that locus $N$ is birational to the Grassmannian $\mathrm{Gr}(5,7)$.
\end{proof}

\begin{remark}\label{rmk: kontsevich dominates M10}
    Notice that in the setting of Proposition \ref{deg4dP4} and Proposition \ref{deg5dP5}, there is a surjective map $\widetilde{M}^{bir}_{1,0}(X, H^3) \to \overline{M}_{1,0}$. This follows from the fact that $\widetilde{M}^{bir}_{1,0}(X, H^3)$ contains nodal rational curves.
\end{remark}

\begin{corollary}
    Let $X$ be a del Pezzo threefold with $\Pic(X) = \BZ H$ and $H^{3}\in \{4,5\}$. For $e\geqslant H^{3}$, there exist elliptic curves of degree $e$ on $X$.
\end{corollary}
\begin{proof}
    When $e = H^{3}$, the claim follows from Proposition \ref{deg4dP4} and Proposition \ref{deg5dP5}. 
    Note that a general member $f\colon E\rightarrow X$ of $\widetilde{M}^{\mathrm{bir}}_{1,0}(X,H^{3})$ is free by the exact sequence $0\rightarrow T_{E}\rightarrow f^{*}T_{X}\rightarrow N_{f}\rightarrow 0$ and the fact that $f(E)$ is a codimension $2$ linear section of $X$.

    Suppose $e>H^{3}$. 
    Gluing a genus one stable map $(g\colon E\rightarrow X)\in \widetilde{M}^{\mathrm{bir}}_{1,0}(X,H^{3})$ and a genus zero stable map $(h\colon \mathbb{P}^{1}\rightarrow X)\in \overline{M}_{0,0}(X,e-H^{3})$, one can obtain a union of free curves $(f\colon E \cup \mathbb{P}^{1} \rightarrow X)\in \overline{M}_{1,0}(X,e)$ with $E\cap \mathbb{P}^{1} = p$.  
    By the exact sequence
    \[
    0\rightarrow f|_{E}^{*}T_{X}(-p)\rightarrow f^{*}T_{X}\rightarrow f|_{\mathbb{P}^{1}}^{*}T_{X}\rightarrow 0, 
    \]
    we obtain that $H^{1}(f^{*}T_{X})= 0$. 
    Hence we can smooth $f$ to a free stable map $\tilde{f}\colon E'\rightarrow X$. 
    By \cite[Proposition 5.5]{Fen26}, we may assume $\tilde{f}$ is balanced. 
    Then $h^{1}(N_{\tilde{f}}(-p-q))=0$ for any $p,q\in E'$. 
    This implies that any fibre of $s\colon M''\rightarrow X\times X$ has the expected dimension $2e-4$, where $M''\subset \widetilde{M}^{\mathrm{bir}}_{1,2}(X,e)$ is the component containing $(\tilde{f}\colon E'\rightarrow X,p,q)$. 
    In particular, the preimage $s^{-1}(\Delta)$ of the diagonal $\Delta\subset X\times X$ cannot form an irreducible component of $\widetilde{M}^{\mathrm{bir}}_{1,0}(X,e)$. 
    Therefore, we conclude that a general deformation of $\tilde{f}$ is an embedding, which proves the claim. 
\end{proof}

\section{Free components}\label{section: free}

In this section, we classify the irreducible components of the Kontsevich moduli space (Theorem \ref{irredKontsevich}). To do so, we will prove the Movable Bend-and-Break (Theorem \ref{Intro:MBB}) and use induction on the irreducibility of lower degree components. This allows us to conclude the uniqueness of Manin components for sufficiently large degree, proving Theorem \ref{irredMor}. The strategies in this section follow closely the arguments in  \cite[Section 5]{Fen26}, hence we focus on the main steps during the proof.

We will denote by $\widetilde{M}^{bir}_{g,n}(X,e)$ the union of irreducible components of $\overline{M}_{g,n}(X,e)$ parametrizing stable maps with irreducible domains that are birational onto the images. We will also abbreviate $H$-degree by degree throughout the section.

\subsection{Movable Bend-and-Break}\label{subsec: MBB}

The idea of Movable Bend-and-Break is simple. Given a free component $M$ of $\widetilde{M}^{bir}_{g,n}(X,e)$, we construct a one-dimensional locus $T$ in $M$ parametrizing stable maps whose domains are irreducible nodal rational curves. Then we apply Mori's Bend-and-Break and classify all possible stable maps with reducible domains in the closure of $T$. To constrain the number of possible outcomes, we require the stable maps parametrized by $T$ to go through the maximum number of general points. This will follow from Theorem \ref{thm: LT22a Prop 2.8} once we know that a general member $[f]$ of $M$ has balanced normal bundle, and this is the content of the following two propositions.

\begin{prop}[cf. \cite{Fen26} Proposition 5.2]\label{prop: locally free normal sheaf}
    Let $X$ be a del Pezzo threefold with $\Pic(X) = \BZ H$ and $H^{3}\in \{4,5\}$. 
    Let $e\geqslant 6$ and let $M\subset \mathrm{Mor}(E,X,e)$ be a dominant component generically parametrizing morphisms that are birational onto their images.
    Then the normal sheaf $N_{f}$ of a general member $f$ of $M$ is locally free. 
\end{prop}
\begin{proof}
    It is enough to show that a general member $f\colon E\rightarrow X$ of $M$ is an immersion. 
    Suppose the contrary. 
    By \cite[1.8 Theorem]{Kollar1996}, we have $h^{1}(E,f^{*}T_{X}(-2p)) > 1$ for some $p\in E$. 
    The general member $f$ of $M$ is free by Theorem \ref{Thm:no non-free component}.  
    Hence $f^{*}T_{X}$ is generically globally generated and the slope panel $SP(f^{*}T_{X}) = (a_{1},a_{2},a_{3})$ with $a_{1}\geqslant a_{2}\geqslant a_{3} > 0$. 
    By the consequences in Section \ref{section: lowdeg}, $a_{1}+a_{2}+a_{3} \geqslant 2e > 6$. 
    If $a_{1}=a_{2}=a_{3}$, then we have $h^{1}(E, f^{*}T_{X}(-2p)) = 0$ since $\deg f^{*}T_{X}(-2p) > 0$. 
    Let $\CG$ be the sum of all summands $\CG'$ of $f^{*}T_{X}$ with $h^{1}(\CG'(-2p))>0$. 
    Then $\deg\CG\leqslant 2\rank\CG\leqslant 6 < \deg f^{*}T_{X}$. 
    Thus $\CG\neq f^{*}T_{X}$. 
    Hence $f^{*}T_{X}$ decomposes into a direct sum of $\CF$ and $\CG$ such that $h^{1}(E, \CF(-2p)) = 0$ and $h^{1}(E, \CG(-2p)) > 1$. 
    
    Assume that $\CG$ is a line bundle. 
    By the Riemann-Roch theorem, we have $h^{1}(E, \CG(-2p)) = h^{0}(E, \CG(-2p))-\chi(E, \CG(-2p)) = h^{0}(E, \CG(-2p)) - \deg\CG + 2$. 
    Hence $h^{1}(E, \CG(-2p))\leqslant 1$. 

    Thus, we see that $\CG$ is a rank two vector bundle. 
    In this case, 
    \[h^{1}(E,\CG(-2p)) = h^{0}(E,\CG(-2p)) -  \chi(E,\CG(-2p)) = h^{0}(E,\CG(-2p)) - \deg\CG + 4.\]
    Hence $\deg\CG \leqslant 4$, and $\deg\CG = 4$ if and only if $\CG = \CO(2p)\oplus \CO(2p)$. 
    Therefore, $a_{2}\leqslant 2$. 
    By the proof of Theorem \ref{no type 2 comp}, we see that $a_{1}-a_{2} < 4$. 
    Hence $\deg f^{*}T_{X} = a_{1} + a_{2} + a_{3} <10$, a contradiction.
\end{proof}

\begin{prop}[cf. \cite{Fen26} Proposition 5.5]\label{Prop: balanced normal bundle}
    Let $X$ be a del Pezzo threefold with $\Pic(X) = \BZ H$ and $H^{3}\in \{4,5\}$. 
    For $e\geqslant H^{3}$, let $M\subset \widetilde{M}_{1,0}^{\mathrm{bir}}(X, e)$ be a free component.
    Then a general member $f$ of $M$ has balanced normal bundle, i.e. $SP(N_{f}) = (e,e)$. 
\end{prop}

\begin{proof}
    By the arguments in \cite[Proposition 5.5]{Fen26}, it suffices to show that any component $M$ of $\widetilde{M}^{bir}_{1,0}(X,e)$ has locally free normal bundle. 
    By Proposition \ref{prop: locally free normal sheaf}, the claim holds for $e\geqslant 6$. 
    For the case $e=H^{3}$, Proposition \ref{deg4dP4} and Proposition \ref{deg5dP5} show that a general member of $\widetilde{M}^{bir}_{1,0}(X,H^{3})$ is a smooth codimension two linear section of $X$, hence they have balanced normal bundles. 
    Thus, the only case left is $\widetilde{M}^{bir}_{1,0}(X,5)$ when $H^{3}=4$. 
    
    By Theorem \ref{no type 2 comp} and Proposition \ref{prop: different type 2 comps}, there is no component in $\Mor(E,X,5)$ which dominates $M$, hence $M$ parametrizes stable maps whose domains vary in the moduli space of genus one curves. Let $M'$ denote the induced component of $M$ in $\widetilde{M}^{bir}_{1,1}(X,5)$ and consider the forgetful map $\pi: M' \to \overline{M}_{1,1}$. A general fibre of $\pi$ is smooth and can be identified with some locus in the morphism scheme $W \subset \Mor(E,X,5)$. This implies that a general member $[f]$ of $W$ has vanishing $h^1$, and similar arguments using slope panel in Proposition \ref{prop: locally free normal sheaf} concludes the result.
\end{proof}

Before proving the Movable Bend-and-Break, we review a useful construction.

\begin{definition} \cite[Section 6.2]{Beheshti2022} 
    Let $p:U\to R$ be a family of irreducible and reduced curves with evaluation map $ev: U\to X$. We say that the family $p$ is basepoint free if $ev$ is flat.
\end{definition}

\begin{construction}(c.f. \cite[Lemma 6.4]{Beheshti2022})
\label{construction of T} 
    Given a free component $M\subset \widetilde{M}^{bir}_{1,0}(X,e)$, we denote by $M^\circ$ the open subset of $M$ parametrizing free curves. We have two possible scenarios for a general member $[f]$ in $M^\circ$:
    \begin{enumerate}
        \item The normal bundle $N_f$ is indecomposable of degree $2e$ or is a direct sum of line bundles of the same degree.
    
        \item The normal bundle $N_f$ is of the form $\CL_1\oplus \CL_2$ where $l_1 = \deg \CL_1$, $l_2 = \deg \CL_2$ and $0 < l_1 < l_2$.
    \end{enumerate}
    Then for any positive integers $r$ and $s$ satisfying (in the respective cases)
    \begin{enumerate}
        \item $r \leqslant e - 1$ and $s \leqslant 2e - 2r - 1$, or
    
        \item $r \leqslant l_1$ and $s \leqslant l_2 - l_1$,
    \end{enumerate}
    denote by $T_{r,s}$ the locus in $M^\circ$ parametrizing curves going through $r$ general points and $s$ general members of $p$. If $p$ is constructed by taking complete intersections of a big and basepoint free linear series, then the proof of \cite[Lemma 6.4]{Beheshti2022} implies that 
    \[
        \codim(T_{r,s}) = 2r+s.
    \]
\end{construction}

We also need the following definition:

\begin{definition}\label{def: banana}
\mbox{}
\begin{itemize}
    \item We say that a prestable curve $C$ is a \textit{banana curve} if $C$ has genus one and consists of two irreducible components $Z_{1}$, $Z_{2}$ which are isomorphic to $\mathbb{P}^{1}$. 

    \item We say that a prestable curve $C$ is of \textit{banana type} if $C$ is obtained from a banana curve by attaching rational curves. 
\end{itemize}
\end{definition}

\begin{theorem}[Movable Bend-and-Break] \label{Thm:MBB} 
    Let $X$ be a smooth del Pezzo threefold with $\Pic(X) = \BZ H$ and $H^{3}\in \{4,5\}$. 
    For $e\geqslant H^{3}+1$, any free component $M\subset\widetilde{M}_{1,0}(X,e)$ contains a stable map $f: Z = Z_1 + Z_2 \to X$ satisfying the following conditions:
    \begin{itemize}
        \item $Z_1$ is a smooth genus one curve and $f|_{Z_1}$ is free.
        \item $Z_2$ is a rational curve and $f|_{Z_2}$ is free.
        \item $f|_{Z_i}$'s are birational onto their images.
    \end{itemize}
\end{theorem}
\begin{proof}
    Let $[f]$ be a general member of $M$. Then $f$ is a birational map onto its image by Proposition \ref{prop: Fen26 Prop 4.6}. By Proposition \ref{Prop: balanced normal bundle}, the normal bundle of $f$ has slope panel $SP(N_{f}) = (e,e)$. Thus Construction \ref{construction of T} gives a $2$-dimensional locus $S\subset M$ consisting of stable maps with irreducible domains whose images pass through $e-1$ general points $p_{1},\dots, p_{e-1}\in X$. Let $\pi\colon \overline{M}_{1,1}(X)\rightarrow \overline{M}_{1,0}(X)$ be the forgetful map, $\rho\colon \overline{M}_{1,1}(X)\rightarrow \overline{M}_{1,1}$ be the stabilization map, and $s\colon \overline{M}_{1,1}(X)\rightarrow X$ be the evaluation map. 
    Since the points $p_{1},\dots,p_{e-1}$ are chosen to be general, $S$ generically parametrizes free stable maps and $\pi^{-1}(S)$ maps surjectively onto $\overline{M}_{1,1}$ by $\rho$. Thus one can take a one parameter family $T\subset S$ consisting of stable maps whose domains are nodal rational curves. Moreover, by \cite[Proposition 5.15]{Fen26}, a general member of $T$ is a smooth point.
    Let $T' \subset \overline{M}_{0,0}(X)$ be the locus obtained by precomposing the stable maps of $T$ with a partial normalization.

    By Mori's Bend-and-Break, there exists a stable map $g'$ with a reducible domain in the closure of $T'$. Since the Movable Bend-and-Break for rational curves holds for del Pezzo threefolds when $e\geqslant 3$ by \cite[Theorem 1.4]{Beheshti2022}, we may assume that $g'$ is a union of two free rational curves $g'\colon Z'_{1}\cup Z'_{2}\rightarrow X$.
    Let $(g\colon Z_{1}\cup Z_{2}\rightarrow X) \in \overline{T}$ be the genus one stable map which is the image of $g'$ under the map $\overline{T}'\rightarrow \overline{T}$, where $\overline{T}'$ (resp. $\overline{T}$) is the closure of $T'$ in $\overline{M}_{0,0}(X)$ (resp. $T$ in $\overline{M}_{1,0}(X)$). Then the same argument as in \cite[Proposition 5.18]{Fen26} shows that there are two possibilities for $g$: 
    \begin{enumerate}[(i)]
        \item $Z_{1}$ is a nodal rational curve and $Z_{2}\cong \mathbb{P}^{1}$:

        By \cite[Theorem 1.4]{Shen12}, the normal bundle of a general map of $R^{rat}_{e'}$ in Theorem \ref{thm: LT19 Thm 7.9} is balanced. Hence a similar computation in the proof of Proposition \ref{prop: locally free normal sheaf} shows that a general member of $R^{rat}_{e'}$ is an embedding. This shows that the locus of degree $e'$ genus one nodal rational curves has dimension at most $2e'-1$. Then Theorem \ref{no type 2 comp}, Proposition \ref{prop: different type 2 comps}, and Remark \ref{rmk: kontsevich dominates M10} together imply that we may smooth $Z_1$ in the domain of $g$ to an elliptic curve $\widetilde{Z}_{1}$ and obtain a stable map $\tilde{g}: \widetilde{Z}_{1}\cup Z_2\to X$. In particular, the restriction $\tilde{g}|_{\widetilde{Z}_{1}}$ is a free genus one curve. This gives the desired stable map. 
        
        \smallskip

        \item The domain of $g$ is a banana curve:

        We may assume that each intersection point of $Z_{1}\cap Z_{2}$ maps to a general point of $X$. Since $e-1\geqslant 3$, the restriction of $g$ to at least one component of the domain is a very free rational curve. 
        Hence by \cite[Lemma 5.29]{Fen26}, we can deform $g$ into a stable map $h\colon Z_{1}\cup Z_{2}\cup L\rightarrow X$ such that $Z_{1}\cup Z_{2}$ is a banana curve and $h|_{L}\in \overline{M}_{0,0}(X,1)$. 
        If $-K_{X}\cdot h_{*}Z_{1} \neq -K_{X}\cdot h_{*}Z_{2}$, which implies $h_{*}Z_{1} \neq h_{*}Z_{2}$, then we can deform $Z_{1}\cup Z_{2}$ into an elliptic curve $\widetilde{Z}_{1}$ and obtain $\widetilde{h}\colon \widetilde{Z}_{1}\cup L\rightarrow X$, as desired.

        Suppose $e_{1} := -K_{X}\cdot h_{*}Z_{1} = -K_{X}\cdot h_{*}Z_{2}$. 
        By assumption, $e_{1}\geqslant 2$ when $H^{3}=4$ and $e_{1}\geqslant 3$ when $H^{3}=5$. 
        Then we see that the evaluation map $s^{(2)}\colon \overline{M}_{0,2}(X,e_{1})\rightarrow X^{2}$ is not birational: indeed, this follows by dimension count except for $H^{3} = 4$ and $e_{1} = 2$. When $H^{3} = 4$ and $e_{1} = 2$, it is known that $s^{(2)}$ is a dominant generically finite morphism of degree $2$ (\cite[Section 3]{Bea95}). Hence one can deform $Z_{1}\cup Z_{2}$ into an elliptic curve $\widetilde{Z}_{1}$ and obtain $\widetilde{h}\colon \widetilde{Z}_{1}\cup L\rightarrow X$, as desired. 

        Finally, a similar reasoning to case (i) shows that the restriction $\widetilde{h}|_{\widetilde{Z_1}}$ is a free genus one curve.
    \end{enumerate}
\end{proof}

\begin{remark}
    Note that Theorem \ref{Thm:MBB} does not hold for $e=5$ and $X$ is a quintic del Pezzo threefold by Proposition \ref{nodeg3or4dP5}. 
    We will explain what will happen when we apply the argument to this case. 
    
    By Proposition \ref{nodeg3or4dP5}, the one parameter family $T$ does not parametrize a union of smooth genus one curve and a rational curve. 
    Hence we find a stable map $f\colon C\rightarrow X$ whose domain is the nodal rational curve. 
    Again by Proposition \ref{nodeg3dP4}, the degenerated stable map for Bend-and-Break fixing two general points will be a banana curve $g\colon Z_{1}\cup Z_{2}\rightarrow X$ with $H\cdot g_{*}Z_{1} = 2$ and $H\cdot g_{*}Z_{2} = 3$. 
    We further break $Z_{2}$ while keeping intersect with $Z_{1}$ into the union $Z'_{1}\cup Z'_{2}$ of a conic and a line. 
    Hence we obtain a stable map $g'\colon Z_{1}\cup Z'_{1}\cup Z'_{2}\rightarrow X$. 
    Now, by construction, $g'(Z_{1})$ and $g'(Z'_{1})$ are conics sharing two general points of $X$. 
    However, there exists a unique conic passing through two general points of $X$ by \cite[Corollary 2.43]{Sanna2014}, i.e. $g'(Z_{1}) = g'(Z'_{1})$. 
    Then smoothing $g'|_{Z_{1}\cup Z'_{1}}$ into a double cover of the conic $g'(Z_{1})$, we finally obtain the union $h\colon Z''_{1}\cup Z'_{2}\rightarrow X$ of the double cover and the line. 
\end{remark}

\subsection{Proof of Main Theorems}\label{subsec: MainThm}

\begin{proof}[Proof of Theorem \ref{irredKontsevich}]
    When the degree $e = 4$, Theorem \ref{hilb scheme of conics} and the same argument of \cite[Theorem 4.7]{Fen26} shows that there exists an irreducible component in $\widetilde{M}_{1,0}(X,e)$ parametrizing double covers of conics. Moreover, this component is unique.
    
    When the degree $e\geqslant 5$, Theorem \ref{coveroflines} and a dimension count imply that the locus of $\widetilde{M}_{1,0}(X,e)$ parametrizing stable maps that are not birational onto their images is exactly the irreducible component $N_{e}$. Hence we consider the locus $\widetilde{M}^{\mathrm{bir}}_{1,0}(X,e)$ parametrizing birational stable maps.
    The base case $e = H^{3}$ is done by Proposition \ref{deg4dP4} and Proposition \ref{deg5dP5}. 
    Suppose $e > H^{3}$. 
    Let $M\subset \widetilde{M}^{\mathrm{bir}}_{1,0}(X,e)$ be an irreducible component. 
    Then there exists a stable map $f\colon Z_{1} \cup Z_{2}\rightarrow X$ as in Theorem \ref{Thm:MBB}. 
    If $f|_{Z_{1}} \notin \widetilde{M}^{\mathrm{bir}}_{1,0}(X,H^{3})$, we may repeatedly apply Theorem \ref{Thm:MBB} to $f|_{Z_{1}}$ to obtain a union $f'\colon Z'_{1} \cup Z'_{2}\rightarrow X$, where $f'|_{Z'}$ is a tree of free rational curves. 
    Then upon smoothing this tree, we obtain a union $g\colon E \cup \mathbb{P}^{1} \rightarrow X$ such that $g|_{E}\in \widetilde{M}^{\mathrm{bir}}_{1,0}(X,H^{3})$ and $g|_{\mathbb{P}^{1}}\in R^{\mathrm{rat}}_{e-H^{3}}$. 

    Now consider the locus $\Delta\subset \widetilde{M}^{\mathrm{bir}}_{1,0}(X,e)$ which is the image of $\widetilde{M}^{\mathrm{bir}}_{1,1}(X,H^{3})\times_{X} (R^{\mathrm{rat}}_{e-H^{3}})'\rightarrow \widetilde{M}^{\mathrm{bir}}_{1,0}(X,e)$, where $(R^{\mathrm{rat}}_{e-H^{3}})'\subset \overline{M}_{0,1}(X, e-H^{3})$ is the irreducible component above $R^{\mathrm{rat}}_{e-H^{3}}\subset \overline{M}_{0,0}(X, e-H^{3})$. 
    By \cite[Lemma 2.1]{Lehmann2021a} and Theorem \ref{thm: LT19 Thm 7.9}, we see that $\Delta$ is irreducible.
    Moreover, since a general member of $\Delta$ is a union of a very free genus one curve and a free rational curve, such a member corresponds to a smooth point of $\widetilde{M}^{\mathrm{bir}}_{1,0}(X,e)$, hence $\Delta$ is contained in a unique component of $\widetilde{M}^{\mathrm{bir}}_{1,0}(X,e)$.
    Since the choice of $M$ is arbitrary, we conclude that $\widetilde{M}^{\mathrm{bir}}_{1,0}(X,e)$ is irreducible. 
\end{proof}

\begin{remark}\label{reducibledomain}
    Note that there exist irreducible components parametrizing stable maps with reducible domains. 
    Let $e, e'\in \mathbb{Z}_{>0}$. Consider the locus of stable maps $(f\colon E\cup \mathbb{P}^{1}\rightarrow X)\in \overline{M}_{1,0}(X,e+e')$ such that $(f|_{E}\colon E\rightarrow X)\in N_{e}$ and $(f|_{\mathbb{P}^{1}}\colon \mathbb{P}^{1}\rightarrow X)\in R^{\mathrm{rat}}_{e'}$. 
    Let $S\subset X$ be the finite subset as in \cite[Theorem 7.6]{Lehmann2019}. 
    Then for each $x\in S$, the locus of the nodal unions of $(f|_{E}\colon E\rightarrow X)\in N_{e}$ and $(f|_{\mathbb{P}^{1}}\colon \mathbb{P}^{1}\rightarrow X)\in R^{\mathrm{rat}}_{e'}$ with $f(E\cap C) = x$ has dimension at most $(1 + 2e) + (2e' - 1) = 2(e + e')$. 
    The locus of the nodal union of $(f|_{E}\colon E\rightarrow X)\in N_{e}$ and $(f|_{\mathbb{P}^{1}}\colon \mathbb{P}^{1}\rightarrow X)\in R^{\mathrm{rat}}_{e'}$ with $f(E\cap C)\notin S$ has dimension $2e + (2e' - 2) + 3 = 2(e + e') + 1$, which is greater than the expected dimension of $\overline{M}_{1,0}(X,e+e')$. 
    Since this locus sweeps out $X$ and its members are not multiple covers of lines, it is an irreducible component not contained in $\widetilde{M}_{1,0}(X,e+e')$.  
\end{remark}

\begin{proof}[Proof of Theorem \ref{irredMor}]
    Let $\Mor^{bir}(E,X,e)$ be the union of irreducible components in $\Mor(E,X)$ parametrizing morphisms of degree $e$ that are birational onto their images. Suppose to the contrary that $\Mor^{bir}(E,X,e)$ is not irreducible for a general $E$.
    
    Let $R'_e \subset \overline{M}_{1,1}(X,e)$ denote the irreducible component corresponding to $R_e$ in Theorem \ref{irredKontsevich}. Consider the morphism 
    \[
    \pi: R'_e \to \overline{M}_{1,1}
    \]
    sending an element $(f: C \to X, p\in C) \in \overline{M}_{1,1}(X,e)$ to $(C, p) \in \overline{M}_{1,1}$. Then the fibre of $\pi$ above $(E,p)$ contains $\Mor^{bir}(E,X,e)$ as an open subscheme. Replace $R'_e$ by a resolution of singularities $\widetilde{R}'_e$ and denote $\widetilde{R}'_e \to \overline{M}_{1,1}$ also by $\pi$. Let the finite part of the Stein factorization of $\pi$ be $\phi: Y \to \overline{M}_{1,1}$. Then the assumption that $\Mor^{bir}(E,X,e)$ is not irreducible implies that $\phi$ is finite of degree greater than $1$. Let $B \subset \overline{M}_{1,1}$ denote the branch locus of $\phi$. Then the fibre of $\pi$ above a point $b\in B$ is generically non-reduced. 

    Let $T$ be the closure of a one-dimensional family of stable maps going through $e-1$ general points and $1$ general member of a basepoint free family constructed in Construction \ref{construction of T}. Then $T$ goes through a non-reduced component $N$ in the fibre above some $b\in B$. Since the support of the branch locus $B$ has length greater than $1$, we may assume that $b$ does not represent the nodal rational curve and hence represents some elliptic curve $E$. 
    
    Let $g: Z\to X$ be a morphism represented by a point in the intersection of $T$ with $N$. Then \cite[Proposition 5.12]{Fen26} implies that there are two possibilities for $g$: either $Z \cong E$, or $Z = E \cup \BP^1$ and $g$ restricts to both components are free. When $Z \cong E$, the component $N$ contains a non-reduced component of $\Mor^{bir}(E, X, e)$, but this contradicts Theorem \ref{Thm:no non-free component}. If $Z$ is reducible, then it suffices to consider the case when $N$ contains some component of $\Mor^{bir}(Z,X,e)$ as an open subset. Since $g|_E$ and $g|_{\BP^1}$ are free curves and it is a codimension one condition for $g|_E$ to intersect the codimension two subscheme $g|_{\BP^1}$ by \cite[3.7 Proposition]{Kollar1996}, the deformations of $g$ in an irreducible component of $\Mor^{bir}(Z,X,e)$ has dimension at most $2e-1$. This contradicts that the dimension of $N$ is at least $2e$. Hence $\Mor^{bir}(E,X,e)$ is irreducible for $E$ general in moduli.  
\end{proof}

\bibliography{math}
\bibliographystyle{alphaurl}

\end{document}